\documentclass[reqno, 11pt]{amsart}
\usepackage{amsmath,mathtools}
 \usepackage{amssymb}
\usepackage{amsthm}
\usepackage{amsmath}
\usepackage{times}
\usepackage{latexsym}
\usepackage[mathscr]{eucal}

\numberwithin{equation}{section}
 
  \newtheorem{theorem}{Theorem}[section]

  \newtheorem{corollary}[theorem]{Corollary}

  \newtheorem{definition}[theorem]{Definition}
  \newtheorem{example}[theorem]{Example}

\title[On contact screen conformal null submanifolds]{On contact screen conformal null  submanifolds}
\author[Samuel Ssekajja]{Samuel Ssekajja*}
\newcommand{\acr}{\newline\indent}
\address{\llap{*\,} School of Mathematics\acr
 University of Witwatersrand\acr
 Private Bag 3, Wits 2050\acr
South Africa}
\email{samuel.ssekajja@wits.ac.za} 
\thanks{}

 \begin{document}
\begin{abstract}
First, we prove that indefinite Sasakian manifolds do not admit any screen conformal $r$-null submanifolds, tangent to the structure vector field. We, therefore,  define a special class of null submanifolds, called; {\it contact screen conformal} $r$-null submanifold of indefinite Sasakian manifolds. Several characterisation results, on the above class of null submanifolds, are proved. In particular, we prove that such null submanifolds exists in indefinite Sasakian space forms of constant holomorphic sectional curvatures of $-3$. 
\end{abstract}
\maketitle
\section{Introduction} 
In the theory of non-degenerate submanifolds, the second fundamental forms and their respective shape operators are related by means of the metric tensor. Contrary to this, there are interrelations between the second fundamental forms of null submanifold  and its screen distribution and their respective shape operators. These interrelations indicates that the null geometry depends on a choice of screen distribution as explained in  \cite{ds2}. While we know that the second fundamental forms of the null submanifolds are independent of a screen (see Theorem 5.1.2 of \cite[p. 199]{ds2}), the same is not true for the fundamental forms of the screens, which is the main cause of non-uniqueness anomaly in the null geometry. Since, in general, it is impossible to remove this anomaly, the authors in \cite{ds2} considered null hypersurfaces and half null submanifolds for which the null and screen second fundamental forms are conformally related. Such classes of null submanifolds are called screen conformal (see Definition 2.2.1 of \cite[p. 51]{ds2} and Definition 4.4.1 of \cite[p. 179]{ds2}). However, this condition can not be used for the case of general $r$-null submanifolds. For this reason, Duggal-Sahin \cite{ds2}, extended the concept of screen conformal to general null submanifolds of semi-Riemannian manifolds (see Definition 5.2.2 of \cite{ds2}).

 In case the  ambient manifold is an indefinite Sasakian manifold, we note that the above concept is not applicable due to some obvious contradictions asproved in Theorem \ref{main1}. Therefore, we introduce the notion of contact screen conformal null submanifolds to cover this gap. Null submanifolds have numerous applications in mathematical physics, particularly in general relativity and electromagnetism, see \cite{db, ds2} for amore details. Many research papers have been published on null geometry, for example \cite{antii, ca, jun, jinduga, db1, Jin, Jin11, dhjin, sam1}, and many more references cited therein. The main objective of this paper is define the concept of contact screen conformal $r$-null submanifolds of indefinite Sasakian manifolds. We prove that such null submanifolds exists in indefinite Sasakian space forms of constant holomorphic sectional curvatures of $-3$. The rest of the paper is arranged as follows; In section \ref{pre} we give basic notions needed in the rest of the paper and in Section \ref{conta}, we present our main results.

\section{Preliminaries} \label{pre}
Let $(\overline{M},\overline{g})$ be a real $(m+n)$-dimensional semi-Riemannian manifold of constant index $\nu$  such that $m,n\ge 1$, $1\le \nu \le m+n-1$, and let $(M,g)$ be an $m$-dimensional submanifold of $\overline{M}$. In case $g$ is \textit{degenerate} on the tangent bundle $TM$ of $M$, we say that $M$ a \textit{null submanifold} \cite{db}. We denote the set of smooth sections of a vector bundle $\Xi$ by $\Gamma(\Xi)$. For a degenerate metric tensor $g=\overline{g}_{|{TM}}$, there exists locally a non-zero vector field $\xi\in \Gamma(TM)$ such that $g(\xi,X)=0$, for any $X\in\Gamma(TM)$. Then, for each tangent space $T_{x}M$, $x\in M$, we have 
$
 T_{x}M^{\perp}=\{u\in T_{x}\overline{M}:\overline{g}(u,v)=0,\; \forall\, v\in T_{x}M\},
$
which is a degenerate $n$-dimensional subspace of $T_{x}\overline{M}$. The \textit{radical} or \textit{null} subspace of $M$ is denoted by $\mathrm{Rad}\,T_{x}M$ and is given by
$\mathrm{Rad}\,T_{x}M=\{\xi_{x}\in T_{x}M: g(\xi_{x},X)=0,\; \forall\, X\in T_{x}M\}$. Notice that $\mathrm{Rad}\,T_{x}M=T_{x}M\cap T_{x}M^{\perp}$ and its dimension depends on $x\in M$. A submanifold $M$ of $\overline{M}$ is called $r$-null if the mapping 
$ 
\mathrm{Rad}\,TM:x\longrightarrow  \mathrm{Rad}\,T_{x}M,
$ 
defines a smooth distribution of rank $r>0$, where $\mathrm{Rad}\,TM$ is called the radical (null) distribution on $M$. Let $S(T M)$ be a \textit{screen distribution} which is a semi-Riemannian complementary distribution of $\mathrm{Rad}\,T M$ in $T M$, and is given by $T M = \mathrm{Rad}\,T M \perp S(T M)$. Note that  the distribution $S(TM)$ is not unique and canonically isomorphic to the factor vector bundle $TM/ \mathrm{Rad}\, TM$  \cite{db}. Choose a screen transversal bundle $S(TM^\perp)$, which is semi-Riemannian complementary to $\mathrm{Rad}\, TM$ in $TM^\perp$. Since, for any local basis $\{\xi_i \}$ of  $\mathrm{Rad}\,TM$, there exists a local null frame $\{N_i\}$ of sections with values in the orthogonal complement of $S(T M^\perp)$ in $S(T M )^\perp$  such that $\overline{g}(\xi_i , N_j ) = \delta_{ij}$, it follows that there exists a {\it null transversal vector bundle } $l\mathrm{tr}(TM)$ locally spanned by $\{N_i\}$ \cite{db}. Let $\mathrm{tr}(TM)$ be complementary (but not orthogonal) vector bundle to $TM$ in $T\overline{M}$. Then, 
\begin{align}\nonumber
  \mathrm{tr}(TM)&=l\mathrm{tr}(TM)\perp S(TM^\perp),\\
  T\overline{M}&= S(TM)\perp S(TM^\perp)\perp\{\mathrm{Rad}\, TM\oplus l\mathrm{tr}(TM)\}\nonumber\\
  &= TM\oplus\mathrm{tr}(TM).\nonumber
\end{align}
 We say that a null submanifold $M$ of $\overline{M}$ is 
 \begin{enumerate}
 	\item $r$-null if $1\leq r< \min\{m,n\}$,
 	\item  co-isotropic if $1\leq r=n<m$, $S(TM^{\perp})=\{0\}$,
  	\item isotropic if $1\leq r=m<n$,  $S(TM)=\{0\}$,
 	\item totally null if $r=n=m$,  $S(TM)=S(TM^\perp)=\{0\}$.
  \end{enumerate}
 Details on the above classes of null submanifolds with examples are found in \cite{db,ds2}. Let $M$ be a coisotropic null submanifold and consider a local quasi-orthonormal fields of frames of $\overline{M}$ along $M$, on $\mathcal{U}$ as 
$
\{ \xi_1,\cdots, \xi_r,N_1,\cdots, N_r,Z_{r+1},\cdots,Z_{m}\},
$
where
 $\{Z_{r+1},\cdots,Z_{m}\}$ is an orthogonal basis of $\Gamma(S(TM)|_{\mathcal{U}})$   and that $\epsilon_a= g(Z_{a},Z_{a})$ is the signature of $\{Z_{a}\}$. The following range of indices will be used. $i,j,k\in\{1,\cdots,r\}$. Let $P$ be the projection morphism of $TM$ onto $S(TM)$. Then, the Gauss-Weingartein equations \cite{ds2} of a coisotropic submanifold $M$ and $S(TM)$ are  
  \begin{align}
      \overline{\nabla}_X Y&=\nabla_X Y+\sum_{i=1}^{r<m} h_i^l(X,Y)N_i,\label{eq11}\\
      \overline{\nabla}_X N_i&=-A_{N_i} X+\sum_{j=1}^{r<m} \tau_{ij}(X) N_{j},\label{eq31}\\
      \nabla_X P Y&=\nabla_X^* PY+\sum_{i=1}^{r<m} h_i^*(X, P Y)\xi_i,\\
      \nabla_X \xi_i &=-A_{\xi_i}^* X-\sum_{j=1}^{r<m} \tau_{ji}(X) \xi_j\label{eq50},
  \end{align}
  for all $X,Y\in \Gamma(TM)$
where  $\nabla$ and $\nabla^*$ are the {\it induced connections} on $TM$ and $S(TM)$ respectively, $h_i^l$'s are symmetric bilinear forms known as \textit{local null fundamental forms} of $TM$. Also, $h_i^*$'s are the \textit{local second fundamental forms} of $S(TM)$. On the other hand, $A_{N_i}$'s and  $A_{\xi_i}^{*}$'s  are linear operators on $TM$ while $\tau_{ij}$'s are 1-forms on $TM$. It is easy to see from (\ref{eq11}) that $h_i^l(X,Y)= \overline{g}(\overline{\nabla}_X Y,\xi_{i})$, for all $X,Y\in \Gamma(TM)$, from which we deduce the independence of $h_i^l$s on  the choice of $S(TM)$. It is easy to see that  $\nabla^*$ is a {\it metric connection} on $S(TM)$ while $\nabla$ is generally not a metric connection and satisfies the relation
      \begin{equation}\label{metric}
         (\nabla_X g)(Y,Z)=\sum_{i=1}^r\{h_i^l(X,Y)\theta_i(Z)+h_i^l(X,Z)\theta_i(Y)\},
      \end{equation}
for any $X,Y\in \Gamma(TM)$ and 1-forms $\theta_i$ given by $\theta_{i}(X)=\overline{g}(X,N_i)$, for all  $X\in \Gamma(TM)$.
     The above two types of local second fundamental forms are related to their shape operators by the following set of equations
      \begin{align}
          g(A_{\xi_i}^*X,Y)&=h_i^l(X,Y)+\sum_{j=1}^rh_j^l(X,\xi_i)\theta_j(Y),\;\;\;\;\bar{g}(A_{\xi_i}^*X,N_j)=0,\label{eqe}\\
         g(A_{N_i}X,Y)&=h_i^*(X,PY),  \;\; \; \theta_{j}(A_{N_{i}}X)+\theta_{i}(A_{N_{j}}X),\;\;X,Y\in \Gamma(TM).\label{K5} 
      \end{align}
 Let $(M,g,S(TM))$ be an $m$-dimensional $r$-null coisotropic submanifold of  $(\overline{M},\overline{g})$. Let $\overline{R}$ and $R$ denote the curvature tensors of $\overline{\nabla}$ and $\nabla$ respectively. The following curvature identities are needed in this paper (see \cite{db} or \cite{ds2} for details) 
\begin{align}
\overline{R}(X,Y)Z=R(X,&Y)Z+\sum_{i=1}^{r}\{h_{i}^{l}(X,Z)A_{N_{i}}Y-h_{i}^{l}(Y,Z)A_{N_{i}}X\}\nonumber\\
&+\sum_{i=1}^{r}[ (\nabla_{X}h_{i}^{l})(Y,Z)-(\nabla_{Y}h_{i}^{l})(X,Z)\nonumber\\
&+\sum_{j=1}^{r}\{\tau_{ji}(X)h_{j}(Y,Z)-\tau_{ji}(Y)h_{j}(X,Z)\}]N_{i},\label{m9}\\
R(X,Y)PZ=R^{*}(&X,Y)PZ+\sum_{i=1}^{r}\{h_{i}^{*}(X,PZ)A_{\xi_{i}}^{*}Y-h_{i}^{*}(Y,PZ)A_{\xi_{i}}^{*}X\}\nonumber\\
&+\sum_{i=1}^{r}[(\nabla_{X}h_{i}^{*})(Y,PZ)-(\nabla_{Y}h_{i}^{*})(X,PZ)\nonumber\\
&+\sum_{j=1}^{r}\{\tau_{ij}(Y)h_{j}^{*}(X,PZ)-\tau_{ij}(X)h_{j}^{*}(Y,PZ)\}]\xi_{i},\label{m10}
\end{align}
for all $X,Y,Z\in\Gamma(TM)$.    An old dimensional smooth manifold $(\overline{M},\overline{g})$ is called a contact metric manifold \cite{ds2} if there exist a (1,1) -tensor field $\overline{\phi}$, a vector field $\zeta$, called the characteristic vector field, and its 1-form $\eta$ satisfying 
  \begin{align}
  	\overline{\phi}^{2}X&=-X+\eta(X)\zeta,\;\; \overline{\phi}\zeta=0, \;\;\; \eta \circ \overline{\phi}=0, \;\;\; \eta(\zeta)=1,\label{m1}\\
  	\overline{g}(\overline{\phi}X,\overline{\phi}Y)&=\overline{g}(X,Y)-\eta(X)\eta(Y),\;\;\; \eta(X)=\overline{g}(\zeta,X),\label{m2}\\
  	&d\eta(X,Y)=\overline{g}(\overline{\phi}X,Y),\;\;\;\forall\, X,Y\in \Gamma(T\overline{M}).\label{{m3}}
  \end{align}
 Then, the set $(\overline{\phi},\eta, \zeta, \overline{g})$  is called a contact metric structure on $\overline{M}$. Furthermore, $\overline{M}$ has a normal contact structure \cite{ds2} if $N_{\overline{\phi}}+2d\eta \otimes \zeta=0$, where $N_{\overline{\phi}}$ is the Nijenhuis tensor field of $\overline{\phi}$. A normal contact metric manifold is called Sasakian \cite{ds2} for which we have 
 \begin{align}\label{m4}
 	(\overline{\nabla}_{X}\overline{\phi})(Y)=\overline{g}(X,Y)\zeta-\eta(Y)X,\;\;\; \forall\, X,Y\in \Gamma(T\overline{M}),
 \end{align}
  where $\overline{\nabla}$ is a metric connection on $\overline{M}$. A Sasakian manifold $\overline{M}=(\overline{M},\overline{\phi}, \zeta,\eta, \overline{g})$ is called an {\it indefinite Sasakian manifold} \cite{ds2} if $(\overline{M},\overline{g})$ is a semi-Riemannian manifold of index $\nu(>0)$. Replacing $Y$ by $\zeta$ in (\ref{m4}), and using (\ref{m1}), we get 
\begin{align}\label{p4}
	\overline{\nabla}_{X}\zeta=-\overline{\phi}X,\;\;\;\,\forall\, X\in \Gamma(T\overline{M}).
\end{align}

A plane section $\pi$ in $T_{x}\overline{M}$ of a Sasakian manifold $\overline{M}$ is called a $\overline{\phi}$-section if it is spanned by a unit vector $X$ orthogonal to $\zeta$ and $\overline{\phi}X$, where $X$ is a non-null vector field on $\overline{M}$. The sectional curvature $K(X,\overline{\phi}X)$ of a $\overline{\phi}$-section is called a $\overline{\phi}$-sectional curvature. If $\overline{M}$ has a $\overline{\phi}$-sectional curvature $c$ which does not depend on the $\overline{\phi}$-section at each point, then, $c$ is constant in $\overline{M}$ and $\overline{M}$ is called a {\it Sasakian space form}, denoted by $\overline{M}(c)$. Moreover, the curvature tensor $\overline{R}$ of $\overline{M}$ satisfies (see \cite{ds2} for more detailes)
\begin{align}\label{p25}
	4\overline{R}(X,Y)Z&=(c+3)\{ \overline{g}(Y,Z)X-\overline{g}(X,Z)Y\}+(c-1)\{ \eta(X)\eta(Z)Y \nonumber\\
	        &-\eta(Y)\eta(Z)X+\overline{g} (X,Z)\eta(Y)\zeta- \overline{g} (Y,Z)\eta(X)\zeta+\overline{g}(\overline{\phi}Y,Z)\overline{\phi}X\nonumber\\
	        &-\overline{g}(\overline{\phi}X,Z)\overline{\phi}Y-2\overline{g}(\overline{\phi}X,Y)\overline{\phi}Z\},\;\;\forall\, X,Y,Z\in \Gamma(T\overline{M}).
\end{align}
 Let $(M,g)$  be a null submanifold of an indefinite Sasakian manifold $(\overline{M},\overline{g})$. If the characteristic vector  field $\zeta$ is tangent to $M$, then it is obvious that $\zeta$ does not belong to $\mathrm{Rad}\, TM$. This enables one to choose a screen distribution $S(TM)$ which contains $\zeta$. This implies that if $\zeta$ is tangent to $M$, then it belongs to $S(TM)$ (see Calin \cite{ca} for more details). Let $M$ be a coisotropic null submanifold of an indefinite Sasakian  $\overline{M}$. Furthermore, we assume that $\overline{\phi}\mathrm{Rad}\, TM$ and $\overline{\phi}l\mathrm{tr}(TM)$ are subbundles of $S(TM)$. It follows that $S(TM)=\{\overline{\phi}\mathrm{Rad}\, TM\oplus \overline{\phi}l\mathrm{tr}(TM)\}\perp D_{0}\perp \mathbb{R}\zeta$, where $D_{0}$ is a non-degenerate almost complex distribution with respect to $\overline{\phi}$, and $\mathbb{R}\zeta$ is a line bundle spanned by $\zeta$. Then, we have $TM=D\oplus D'\perp \mathbb{R}\zeta$, where $D=\mathrm{Rad}\, TM\perp\overline{\phi}\mathrm{Rad}\, TM\perp D_{0}$ and $D'=\overline{\phi}l\mathrm{tr}(TM)$. Consider local null vector fields $U_{i}$,$V_{i}$, for each $i\in \{1,\ldots, r\}$, and their 1-forms $u_{i}$,  $v_{i}$ defined by 
 \begin{align}
 	U_{i}&=-\overline{\phi}N_{i},\;\;\;\;\;\;\;\;\;\;\;\;\; V_{i}=-\overline{\phi}\xi_{i},\label{m6}\\
 	u_{i}(X)&=g(X,V_{i}),\;\;\;\; v_{i}(X)=g(X,U_{i}).\label{m7}
 \end{align}
 Let $S$ be the projection morphism of $TM$ onto $D$. Then, for any $X\in \Gamma(TM)$,
 \begin{align}\label{m8}
 	\overline{\phi}X=\phi X+\sum_{i=1}^{r}u_{i}(X)N_{i},
 \end{align}
where $\phi$ is a tensor field of type (1,1) globally defined on $M$ by $\phi=\overline{\phi}\circ S$. By a direct calculation using (\ref{eq11})--(\ref{eq50}), (\ref{m6})--(\ref{m8}), we derive 
\begin{align}
	h_{j}^{l}(X,U_{i})&=h_{i}^{*}(X,V_{j}),\;\;\; h_{j}^{l}(X,V_{i})=h_{i}^{l}(X,V_{j}),\label{m11}\\
	\nabla_{X}U_{i}&=\phi A_{N_{i}}X+\sum_{j=1}^{r}\tau_{ij}(X)U_{j}-\theta_{i}(X)\zeta,\label{m12}\\
	\nabla_{X}V_{i}&=\phi A_{\xi_{i}}^{*}X-\sum_{j=1}^{r}\tau_{ji}(X)V_{j}+\sum_{j=1}^{r}h_{j}^{l}(X,\xi_{i})U_{j},\label{m15}
\end{align}
 for all $X,\in \Gamma(TM)$. On the other hand, using (\ref{p4}), (\ref{eq11})--(\ref{eq50}), we have 
 \begin{align}
 	h_{i}^{l}(X,\zeta)=-u_{i}(X),\;\;\;\;\; h_{i}^{*}(X,\zeta)=-v_{i}(X),\label{m14}
 \end{align}
 for all $X\in \Gamma(TM)$.

\section{Contact screen conformal submanifolds}\label{conta}

On a null hypersurface of a semi-Riemannian manifold, the screen and null shape operators $A_{\xi}^{*}$ and $A_{N}$, respectively, where $\xi\in \Gamma(TM^{\perp})$ and $N\in \Gamma(tr(TM))$, are both screen-valued operators. Due to this fact, it is always possible to link the two operators via a non-vanishing smooth function to give rise to a class of hypersurfaces called; {\it screen conformal null hypersurfaces} (see \cite[Definition 2.2.1]{ds2}).  A similar consideration is done for half null submanifolds (see \cite[Definition 4.4.1]{ds2}). However, this cannot be considered for a general null submanifold due to the fact that the shape operators $A_{N_{i}}$, for all $i\in \{1\ldots, r\}$, are generally not screen-valued. For the above reason, Duggal-Sahin \cite{ds2} defined a certain type of screen conformality for a coisotropic submanifold as follows;
\begin{definition}[\cite{ds2}]\label{def1}
	\rm{
A coisotropic null submanifold $(M,g, S(TM))$	 of a semi-Riemannian manifold $(\overline{M},\overline{g})$ is called a screen locally conformal submanifold if the fundamental forms $h_{i}^{*}$ of $S(TM)$ are conformally related to the corresponding null fundamental forms $h_{i}^{l}$ of $M$ by 
\begin{align}\label{m17} 
	h_{i}^{*}(X,PY)=\varphi_{i}h_{i}^{l}(X,Y),\;\;\;\forall\, X,Y\in \Gamma(TM), \;\; i\in \{i,\ldots, r\},
\end{align}
where the $\varphi_{i}'$s are smooth functions on the neighbourhood  $\mathcal{U}$ of $M$.
}
	\end{definition}
 It is then proved, in Theorem 5.2.3 of \cite[p. 204]{ds2},  that any screen distribution satisfying Definition \ref{def1} is integrable. Further still, the definition is extended to $r$-null submanifolds in which similar conclusions on the integrability of $S(TM)$ are reached (see Theorem 5.2.6 of \cite[p. 219]{ds2}).
\begin{example}[\cite{ds2}] 
\rm{
Let us consider the coisotropic submanifold $x_{2}=(x_{3}^{2}+x_{5}^{2})^{1/2}$, $x_{4}=x_{1}$, $x_{3}, x_{5}>0$ of $\overline{M}=(\mathbb{R}_{2}^{5},\overline{g})$, where $\mathbb{R}_{2}^{5}$ is a semi-Euclidean space of signature $(-,-,+,+.+)$ with respect to the canonical basis $(\partial x_{1}, \partial x_{2}, \partial x_{3}, \partial x_{4}, \partial x_{5})$. Then, it is easy to check that $S(TM)=\mathrm{Span}\{X\}$, $\mathrm{Rad}\, TM=\mathrm{Span}\{\xi_{1},\xi_{2}\}$ and $l\mathrm{tr}(TM)=\mathrm{Span}\{N_{1},N_{2}\}$, where $X=x_{5}\partial {x_{2}}+x_{2}\partial x_{5}$, $\xi_{1}=\partial_{1}+\partial_{4}$, $\xi_{2}=x_{2}\partial x_{2}+x_{3}\partial x_{3}+ x_{5}\partial_{5}$, $N_{1}=(1/2)(-\partial x_{1}+\partial x_{4})$  and $N_{2}=(1/2x_{3}^{2})(-x_{2}\partial x_{2}+x_{3}\partial x_{3}-x_{5}\partial x_{5})$. Then, a direct a direct calculation reveals that $\overline{\nabla}_{\xi_{1}}X=0$, $\overline{\nabla}_{\xi_{2}}X=X$, $\overline{\nabla}_{\xi_{1}}\xi_{2}=0$ and $\overline{\nabla}_{X}X=x_{2}\partial x_{2}+x_{5}\partial x_{5}$. Next, by Gauss' formulae, we have $\nabla_{X}X=(1/2)\xi_{2}$, $h_{1}^{*}(X,X)=0$, $h_{1}^{*}(\xi_{1},X)=h_{2}^{*}(\xi_{1},X)=0$, $h_{1}^{*}(\xi_{2},X)=h_{2}^{*}(\xi_{2},X)=h_{1}^{l}(\xi_{1},X)=h_{2}^{l}(\xi_{2},X)=0$, $h_{1}^{l}=0$, $h_{2}^{l}(X,X)=-x_{3}^{2}$ and $h_{2}^{*}=1/2$. It follows that $M$ is screen conformal with $\varphi_{1}$ arbitrary and $\varphi_{2}=-1/2x_{3}^{2}$.
}
\end{example}
 However, it is very important to note that when the ambient space is an indefinite Sasakian manifold, such screen conformal null submanifolds, tangent to $\zeta$, i.e. $\zeta\in \Gamma(TM)$, do not exist. In fact, we have have the following result.
\begin{theorem}\label{main1}
	There does not exist any screen locally conformal null subamnifolds $(M,g, S(TM))$, tangent to the structure vector field $\zeta$, of an indefinite Sasakian manifold $(\overline{M},\overline{g})$. 
\end{theorem}
\begin{proof}
	Assume, on contrary, that $M$ is locally screen conformal, then from (\ref{m14}) and (\ref{m17}) of Definition \ref{def1}, we have 
	\begin{align}\label{m18}
		-v_{i}(X)= h_{i}^{*}(X,\zeta)=\varphi_{i}h_{i}^{l}(X,\zeta)=-\varphi_{i}u_{i}(X),
	\end{align}
for all $X\in \Gamma(TM)$. Setting $X=V_{j}$ in (\ref{m18}) we get $-v_{i}(V_{j})=-\delta_{ij}=0$, for all $i,j \in \{1,\ldots,r\}$, which is a contradiction. On the other hand letting $X=U_{j}$ in (\ref{m18}), we get $-\varphi_{i} u_{i}(U_{j})=0$. As $\varphi_{i}$'s are nonzero, it follows that $-u_{i}(U_{j})=-\delta_{ij}=0$, which is also a contradiction. Hence, $M$ can not be locally screen conformal in an indefinite Sasakian manifold.
\end{proof}
Based on Theorem \ref{main1}, we notice that Definition \ref{def1} fails for null submanifolds of indefinite Sasakian manifolds, mainly in portions of $TM$ containing the structure vector field $\zeta$. This can be rectified by defining the concept of screen conformality of $h_{i}^{*}$ and $h_{i}^{l}$ on $D\oplus D'$, instead of $TM=D\oplus D'\perp \mathbb{R}\zeta$. To this end, let $\tilde{P}$ be the projection morphism of $TM$ onto the subbundle $D\oplus D'$. It then follows easily that any $X\in \Gamma(TM)$ can be written as $X=\tilde{P}X+\eta(X)\zeta$. Then, by a direct calculation, we have 
\begin{align}\label{m20}
h_{i}^{*}(\tilde{P}X,\tilde{P}PY)
&=h_{i}^{*}(X,Y)-\eta(Y)h_{i}^{*}(X,\zeta)-\eta(X)h_{i}^{*}(\zeta,Y),	
\end{align}
for all $X,Y\in \Gamma(TM)$, in which we have used (\ref{m14}) to deduce that $h_{i}^{*}(\zeta,\zeta)=-v_{i}(\zeta)=0$. We also have, 
\begin{align}\label{m21}
	h_{i}^{l}(\tilde{P}X,\tilde{P}Y)
	&=h_{i}^{l}(X,Y)-\eta(Y)h_{i}^{l}(X,\zeta)-\eta(X)h_{i}^{l}(\zeta,Y),
\end{align}
for all $X,Y\in \Gamma(TM)$, in which we have used the fact that $h_{i}^{l}(\zeta,\zeta)=-u_{i}(\zeta)=0$. 

Then, we have the following definition; 
\begin{definition}
	\rm{
	Let $(\overline{M},\overline{\phi}, \zeta,\eta, \overline{g})$ be an indefinite almost contact manifold. A null submanifold $(M,g,S(TM))$, tangent to the structure vector field $\zeta$, is called contact locally screen conformal if the fundamental forms $h_{i}^{*}$ of $S(TM)$ are conformally related to the corresponding null fundamental forms $h_{i}^{l}$ of $M$, on the subbundle $D\oplus D'$ by 
\begin{align}\label{m22} 
	h_{i}^{*}(\tilde{P}X,\tilde{P}PY)=\varphi_{i}h_{i}^{l}(\tilde{P}X,\tilde{P}Y),\;\;\;\forall\, X,Y\in \Gamma(TM), \;\; i\in \{i,\ldots, r\},
\end{align}
where the $\varphi_{i}'$s are smooth functions on the neighbourhood  $\mathcal{U}$ of $M$. In view of (\ref{m14}), (\ref{m20}) and (\ref{m21}), $M$ is contact locally screen conformal if 
\begin{align}\label{m30}
	h_{i}^{*}(X,PY)=\varphi_{i}\{h_{i}^{l}&(X,Y)+u_{i}(X)\eta(Y)+u_{i}(Y)\eta(X)\}\nonumber\\
	&-v_{i}(X)\eta(Y)+h_{i}^{*}(\zeta,Y)\eta(X),
\end{align}
for all $X,Y\in \Gamma(TM)$.
}
\end{definition}
It has has been established (see Theorem 5.2.3 of \cite[p. 204]{ds2}) that when $M$ is locally screen conformal then $S(TM)$ is integrable. However, this is not generally true for a contact locally screen conformal null submanifold. In fact, we have the following result.
\begin{theorem}
	The screen distribution $S(TM)$ of a contact locally screen conformal null submanifold of an indefinite Sasakian manifold is integrable if and only if $h^{*}_{i}(\zeta,PX)=-v_{i}(X)$, for all $X\in \Gamma(TM)$.
\end{theorem}
\begin{proof}
	In view of (\ref{m30}) and the symmetry of $h_{i}^{l}$'s, we have 
	\begin{align}\label{m31}
		h_{i}^{*}(X,Y)-h_{i}^{*}(Y,X)&=v_{i}(Y)\eta(X)-v_{i}(X)\eta(Y)\nonumber\\
		&+h_{i}^{*}(\zeta,Y)\eta(X)-h_{i}^{*}(\zeta,X)\eta(Y),
	\end{align}
	for all $X,Y\in \Gamma(S(TM))$. Now, if $S(TM)$ is integrable then all $A_{N_{i}}$ are symmetric on $S(TM)$ by Theorem 5.1.5 of \cite{ds2}. Hence, the left hand side of (\ref{m31}) vanishes which further implies that $v_{i}(Y)\eta(X)-v_{i}(X)\eta(Y)+h_{i}^{*}(\zeta,Y)\eta(X)-h_{i}^{*}(\zeta,X)\eta(Y)=0$, for all $X,Y\in \Gamma(TM)$. Letting $Y=\zeta$ in this relation we get $-v_{i}(X)+h_{i}^{*}(\zeta,\zeta)\eta(X)-h_{i}^{*}(\zeta,X)=0$. But, by (\ref{m14}), we see that $h_{i}^{*}(\zeta,\zeta)=-v_{i}(\zeta)=0$. Hence, we have $h_{i}^{*}(\zeta,X)=-v_{i}(X)$. The converse is obvious, which completes the proof. 
\end{proof}
\begin{theorem}\label{main4}
	Let $(M,g)$ be a contact screen conformal coisotropic null submanifold, tangent to $\zeta$, of an indefinite Sasakian space form $\overline{M}(c)$. Then, $c=-3$. Moreover, the functions $\varphi_{i}$, $i\in \{i,\ldots, r\}$, satisfies the  differential equations 
	\begin{align}\label{m40}
		(\xi_{j}\varphi_{i})h_{i}^{l}(V_{i},U_{i})-\sum_{k=1}^{r}[\varphi_{k}\tau_{ik}(\xi_{j})+\varphi_{i}\tau_{ki}(\xi_{j})]h_{k}^{l}(V_{i},U_{i})=0,
	\end{align}
	and the curvature tensor of $(M,g)$ takes the form 
	\begin{align}\label{m43}
		R(X,Y)Z&=-\sum_{i=1}^{r}\{h_{i}^{l}(X,Z)A_{N_{i}}Y-h_{i}^{l}(Y,Z)A_{N_{i}}X\}- \eta(X)\eta(Z)Y \nonumber\\
	        &+\eta(Y)\eta(Z)X-\overline{g} (X,Z)\eta(Y)\zeta+ \overline{g} (Y,Z)\eta(X)\zeta-\overline{g}(\overline{\phi}Y,Z)\phi X\nonumber\\
	        &+\overline{g}(\overline{\phi}X,Z)\phi Y+2\overline{g}(\overline{\phi}X,Y)\phi Z, \;\;\;\forall\,X,Y,Z\in \Gamma(TM).
	\end{align}
\end{theorem}	
\begin{proof}
	For all $X,Y,Z\in \Gamma(D\oplus D')$, relations (\ref{m14}) and  (\ref{m22}) leads to 
	\begin{align}\label{m50}
		(\nabla_{X}&h_{i}^{*})(Y,PZ)=(X\varphi_{i})h_{i}^{l}(Y,PZ)+\varphi_{i}\{Xh_{i}^{l}(Y,PZ)-h_{i}^{l}(\tilde{P}\nabla_{X}Y,PZ)\nonumber\\
		&- h_{i}^{l}(Y,\tilde{P}\nabla_{X}^{*}PZ)\}-h_{i}^{*}(\zeta,PZ)\eta(\nabla_{X}Y)+v_{i}(Y)\eta(\nabla_{X}^{*}PZ).
	\end{align}
	On the other hand, using (\ref{m50}), we have 
	\begin{align}\label{m51}
		(\nabla_{X}h_{i}^{l})(Y,PZ)&=Xh_{i}^{l}(Y,PZ)-h_{i}^{l}(\tilde{P}\nabla_{X}Y,PZ)-h_{i}^{l}(Y,\tilde{P}\nabla_{X}^{*}PZ)\nonumber\\
		&+u_{i}(PZ)\eta(\nabla_{X}Y)+u_{i}(Y)\eta(\nabla_{X}PZ),
	\end{align}
	for all $X,Y,Z\in \Gamma(D\oplus D')$. Using (\ref{m50}) and (\ref{m51}), we derive 
	\begin{align}\label{m54}
		(\nabla_{X}h_{i}^{*})(Y,PZ)&=(X\varphi_{i})h_{i}^{l}(Y,PZ)+\varphi_{i}\{ (\nabla_{X}h_{i}^{l})(Y,PZ)\nonumber\\
		&-u_{i}(PZ)\eta(\nabla_{X}Y)-u_{j}(Y)\eta(\nabla_{X}^{*}PZ)\}\nonumber\\
		&-h_{i}^{*}(\zeta,PZ)\eta(\nabla_{X}Y)+v_{i}(Y)\eta(\nabla_{X}^{*}PZ).
	\end{align}
	Interchanging $X$ and $Y$ in (\ref{m54}), subtracting the two relations and then use (\ref{m9}) and (\ref{m10}), we get
	\begin{align}\label{m59}
		\overline{g}(\overline{R}&(X,Y)PZ,N_{i})-\varphi_{i}\overline{g}(\overline{R}(X,Y)PZ,\xi_{i})\nonumber\\
		&=(X\varphi_{i})h_{i}^{l}(Y,PZ)-(Y\varphi_{i}) h_{i}^{l}(X,PZ)+\varphi_{i}\{ u_{i}(PZ)\eta(\nabla_{Y}X)\nonumber\\
		&-u_{i}(PZ)\eta(\nabla_{X}Y)+u_{i}(X)\eta(\nabla_{Y}^{*}PZ)-u_{i}(Y)\eta(\nabla_{X}^{*}PZ)\}\nonumber\\
		&+h_{i}^{*}(\zeta,PZ)\eta(\nabla_{Y}X)-h_{i}^{*}(\zeta,PZ)\eta(\nabla_{X}Y)+v_{i}(Y)\eta(\nabla_{X}^{*}PZ)\nonumber\\
		&-v_{i}(X)\eta(\nabla_{Y}^{*}PZ)+\sum_{j=1}^{r}\varphi_{j}\{ h_{j}^{l}(X,PZ)\tau_{ij}(Y)-h_{j}^{l}(Y,PZ)\tau_{ij}(X)\}\nonumber\\
		&-\sum_{j=1}^{r}\varphi_{i}\{h_{j}(Y,PZ)\tau_{ji}(X)-h_{j}^{l}(X,PZ)\tau_{ji}(Y)\}.
	\end{align}
	Then, applying (\ref{p25}) to (\ref{m59}) and then let $X=\xi_{k}$, we derive
	\begin{align}\label{m60}
		\frac{c+3}{4}&g(Y,PZ)\delta_{ik}+\frac{c-1}{4}\{u_{k}(PZ)v_{i}(Y)+2u_{k}(Y)v_{i}(PZ)\}\nonumber\\
		&\frac{c-1}{4}\varphi_{i}\{u_{k}(PZ)u_{i}(Y)+2u_{k}(Y)u_{i}(PZ)=(\xi_{k}\varphi_{i})h_{i}^{l}(Y,PZ)\nonumber\\
		&+\varphi_{i}\{u_{i}(PZ)u_{k}(Y)-u_{i}(PZ)\eta(\nabla_{\xi_{k}}Y)-u_{i}(Y)\eta(\nabla^{*}_{\xi_{k}}PZ)\}\nonumber\\
		&+h_{i}^{*}(\zeta,PZ)u_{k}(Y)-h_{i}^{*}(\zeta,PZ)\eta(\nabla_{\xi_{k}}Y)+v_{i}(Y)\eta(\nabla_{\xi_{k}}^{*}PZ)\nonumber\\
		&-\sum_{j=1}^{r}\varphi_{j}h_{j}^{l}(Y,PZ)\tau_{ij}(\xi_{k})-\sum_{j=1}^{r}\varphi_{i}h_{j}^{l}(Y,PZ)\tau_{ji}(\xi_{k}),
	\end{align}
	for all $Y,Z\in \Gamma(D\oplus D')$. Interchanging $j$ and $k$ in (\ref{m60}) and then substitute $Y=V_{\ell}$ and $PZ=U_{\ell}$, we get 
	\begin{align}\label{m62}
		\frac{c+3}{4}\delta_{ij}&+\frac{c-1}{4}\delta_{i\ell}\delta_{j\ell}=(\xi_{j}\varphi_{i})h_{i}^{l}(V_{\ell},U_{\ell})-\varphi_{i}\eta(\nabla_{\xi_{j}}V_{\ell})\delta_{i\ell}\nonumber\\
		&-h_{i}^{*}(\zeta,U_{\ell})\eta(\nabla_{\xi_{j}}V_{\ell})+\eta(\nabla_{\xi_{j}}U_{\ell})\delta_{i\ell}-\sum_{k=1}^{r}\varphi_{k}h_{k}^{l}(V_{\ell},U_{\ell})\tau_{ik}(\xi_{j})\nonumber\\
		&-\sum_{k=1}^{r}\varphi_{i}h_{k}^{l}(V_{\ell},U_{\ell})\tau_{ki}(\xi_{j}).
	\end{align}
	But, in view of (\ref{m12}) and (\ref{m15}), we have 
	\begin{align} 
		\eta(\nabla_{\xi_{j}}U_{\ell})&=\eta(\phi A_{N_{\ell}}\xi_{j})+\sum_{k=1}^{r}\tau_{\ell k}(\xi_{j})\eta(U_{k})-\theta_{\ell}(\xi_{j})\eta(\zeta)=-\delta_{\ell j},\label{m70}\\
		\eta(\nabla_{\xi_{j}}V_{\ell})&=\eta(\phi A^{*}_{\xi_{\ell}}\xi_{j})-\sum_{k=1}^{r}\tau_{k \ell }(\xi_{j})\eta(V_{k})+\sum_{k=1}^{r}h_{k}^{l}	(\xi_{j},\xi_{\ell})\eta(U_{k})=0.\label{m71}	
	\end{align}
	Replacing (\ref{m70})  and (\ref{m71}) in (\ref{m62}), we get 
	\begin{align}
		\frac{c+3}{4}\{\delta_{ij}+\delta_{\ell i}\delta_{\ell j}\}&=(\xi_{j}\varphi_{i})h_{i}^{l}(V_{\ell},U_{\ell})-\sum_{k=1}^{r}\varphi_{k}h_{k}^{l}(V_{\ell},U_{\ell})\tau_{ik}(\xi_{j})\nonumber\\
		&-\sum_{k=1}^{r}\varphi_{i}h_{k}^{l}(V_{\ell},U_{\ell})\tau_{ki}(\xi_{j}).\label{m72}
	\end{align}
	Setting $i=\ell$ in (\ref{m72}), we get 
	\begin{align}
		\frac{c+3}{2}\delta_{ij}&=(\xi_{j}\varphi_{i})h_{i}^{l}(V_{i},U_{i})-\sum_{k=1}^{r}\varphi_{k}h_{k}^{l}(V_{i},U_{i})\tau_{ik}(\xi_{j})\nonumber\\
		&-\sum_{k=1}^{r}\varphi_{i}h_{k}^{l}(V_{i},U_{i})\tau_{ki}(\xi_{j}).\label{m73}
\end{align}
On the other hand, if we set $Y=U_{\ell}$ and $PZ=V_{\ell}$ in (\ref{m60}), and then following the simplifications in (\ref{m62})-(\ref{m73}), we have 
\begin{align}
		\frac{3}{4}(c+3)\delta_{ij}&=(\xi_{j}\varphi_{i})h_{i}^{l}(V_{i},U_{i})-\sum_{k=1}^{r}\varphi_{k}h_{k}^{l}(V_{i},U_{i})\tau_{ik}(\xi_{j})\nonumber\\
		&-\sum_{k=1}^{r}\varphi_{i}h_{k}^{l}(V_{i},U_{i})\tau_{ki}(\xi_{j}).\label{m74}
\end{align}
Then, from (\ref{m73}) and (\ref{m74}), we have $\frac{c+3}{4}=0$ or simply $c=-3$. Moreover, we also have $$
(\xi_{j}\varphi_{i})h_{i}^{l}(V_{i},U_{i})-\sum_{k=1}^{r}\varphi_{k}h_{k}^{l}(V_{i},U_{i})\tau_{ik}(\xi_{j})-\sum_{k=1}^{r}\varphi_{i}h_{k}^{l}(V_{i},U_{i})\tau_{ki}(\xi_{j})=0,
$$
which proves (\ref{m40}). Finally, (\ref{m43}) follows from (\ref{m9}) and (\ref{p25}) with $c=-3$, which completes the proof.
\end{proof}	
The following result also follows from Theorem \ref{main4}.	
\begin{corollary}
	There does not exist any contact screen conformal coisotropic submanfold, tangent to $\zeta$, of an indefinite Sasakian space form $\overline{M}(c\ne -3)$.
\end{corollary}

\end{document}